\documentclass{amsart}

\usepackage{hyperref}
\usepackage{blindtext}
\usepackage{subfiles} 

\usepackage{commonpackage}

\usepackage{macro}

\numberwithin{equation}{section}
\theoremstyle{plain}
\newtheorem{theorem}{Theorem}[section]

\newtheorem{claim}[theorem]{Claim}

\newtheorem{conjecture}[theorem]{Conjecture}

\newcommand{\Lam}{\Lambda}

\newcommand{\e}{\epsilon}

\title{A restricted projection problem for fractal sets in $\R^n$}
\date{}

\author[Gan]{Shengwen Gan}
\address{Department of Mathematics\\
Massachusetts Institute of Technology\\
Cambridge, MA 02142-4307, USA}
\email{shengwen@mit.edu}

\author[Guo]{Shaoming Guo}
\address{Department of Mathematics\\
University of Wisconsin-Madison\\
Madison, WI-53706, USA}
\email{shaomingguo@math.wisc.edu}

\author[Wang]{Hong Wang}
\address{Department of Mathematics\\
UCLA\\
Los Angeles, CA 90095, USA}
\email{hongwang@math.ucla.edu}

\begin{document}

\maketitle
\begin{abstract}
    Let $\bgamma: [-1, 1]\to \R^n$ be a smooth curve that is non-degenerate. Take $m\le n$ and a Borel set $E\subset [0, 1]^n$. We prove that the orthogonal projection of $E$ to the $m$-th order tangent space of $\bgamma$ at $\theta\in [-1, 1]$ has Hausdorff dimension $\min\{m, \dim(E)\}$ for almost every $\theta\in [-1, 1]$. 
\end{abstract}

\section{Introduction}

Let $V\subset \mathbb{R}^n$ be a $m$--dimensional subspace, the orthogonal projections $\pi_{V}$ is defined as 
\[
P_{V}: \mathbb{R}^n\rightarrow \mathbb{R}^m, \quad P_V(x):= v
\]
where $\mathbb{R}^n= V \bigoplus U$ and $x=v+u$ is the unique decomposition with $v\in V$ and $u\in U$. 

Given a Borel subset $\Lambda \subset Gr(m, \mathbb{R}^n)$, the Grassmannian of $m$-dimensional linear subspaces, and a compact Borel  set $E\subset \mathbb{R}^n$, a general question about orthogonal projection is  to determine 
\[
\sup_{V\in \Lambda} \dim (P_V(E))
\]
in terms of $\dim (\Lambda)$ and $\dim (E)$.   The classical Marstrand's projection theorem (\cite{Mar54}, \cite{Mat15}) states that 
\begin{equation}\label{marstrand}
    \dim (P_{V} (E))  = \min \{\dim(E), m\} \textup{~for~a.e.} \  V\in Gr(m,\R^n), 
\end{equation}
The question for general Borel subset of subpaces $\Lambda \subset Gr(m, \mathbb{R}^n)$ was later studied intensively, including 
Mattila \cite{Mat75}, Falconer \cite{Fal80}, Bourgain \cite{Bou10}, Peres and Schlag \cite{PS00}. Analogous to \eqref{marstrand}, one may conjecture that
\begin{equation}\label{conje}
     \dim (P_{V} (E))  = \min \{\dim (E), m\} \textup{~for~a.e.} \  V\in\Lambda,
\end{equation}
where ``a.e." is with respect to certain probability measure supported on $\Lam$, chosen naturally depending on the choice of $\Lambda$.

However, \eqref{conje} fails if we do not make further assumptions on $\Lam$. We discuss some examples in $\R^3$. The examples in higher dimensions can be constructed in the same way.
When $m=1$, let $E$ be the $x_3$-axis and let $\Lambda$ consist of lines in the $(x_1,x_2)$-plane. We see that $\dim(E)=1$, while $\dim(P_V(E))=0$ for all $V\in\Lambda$ which violates \eqref{conje}. When $m=2$, let $E$ be the $(x_1,x_2)$-plane and let $\Lambda$ consist of planes containing the $x_3$-axis. We see that $\dim(E)=2$, while $\dim(P_V(E))=1$ for all $V\in\Lambda$ which violates \eqref{conje}. It is natural to ask whether one can quantitatively exclude the above degenerate scenario and provide good estimates about $\sup_{V\in \Lambda}\dim (P_V(E))$.

One such example is the following conjecture by F\"{a}ssler and Orponen \cite{FO14}. 
\begin{conjecture}[F\"{a}ssler-Orponen]\label{conj: restricted proj}
Let $I\subset \mathbb{R}$ be a compact interval and let $\gamma: I\rightarrow \mathbb{S}^2$ be a $C^2$ curve that ``escapes the great circle''
\begin{equation}\label{240120non_degenerate}
\text{span}\{\gamma(\theta), \gamma'(\theta), \gamma''(\theta)\} =\mathbb{R}^3 \quad \text{for all }\theta\in I.
\end{equation}
Let $E\subset \mathbb{R}^3$ be analytic. Then for almost every $\theta\in I$, 
\begin{enumerate}
\item [(a)] $\dim (P_{\gamma(\theta)} E) =\min \{ \dim (E), 1\}.$
\item [(b)] $\dim  (P_{ \gamma(\theta)^{\perp}} E) =\min \{\dim (E), 2\}.$
\end{enumerate}
\end{conjecture}

This conjecture was studied by F\"assler and Orponen \cite{FO14}, Orponen \cite{Orp15}, Oberlin and Oberlin \cite{OO15}, Orponen and Venieri \cite{OV20} and Harris \cite{Har19, Har21, Har22}. 
Part (a) of Conjecture~\ref{conj: restricted proj} was resolved by K\"aenm\"aki, Orponen and Venieri \cite{KOV17} in the case where $\gamma$ is a small circle, and by  Pramanik, Yang and Zahl \cite{PYZ22} for the case of general $C^2$ curves, and also later by Gan, Guth and Maldague \cite{GGM22} independently. Indeed these papers prove something much stronger, in the sense that they contain very good exceptional set estimates. Part (b) of Conjecture~\ref{conj: restricted proj} was resolved by Gan, Guo, Guth, Harris, Maldague and Wang \cite{GGG22}.

The restricted projection problem has applications in ergodic theory, see the work of Lindenstrauss and Mohammadi \cite{LM22}, Lindenstrauss, Mohammadi, Wang and Yang \cite{LMWY23}, and a special case of the Kakeya problem, see recent works of F\"{a}ssler and Orponen \cite{FO22} and Katz, Wu and Zahl \cite{KWZ22}. \\

We prove a generalization of Conjecture~\ref{conj: restricted proj} for higher dimensions. 
In $\R^n$, let
\begin{equation}
    \bgamma: [-1, 1]\to \R^n
\end{equation}
be a smooth non-degenerate curve, that is, 
\begin{equation}
    \det[\bgamma'(\theta), \dots, \bgamma^{(n)}(\theta)]\neq 0,
\end{equation}
for every $\theta\in [-1, 1]$. 
For $1\le m\le n$, define the $m$-th order tangent space of $\bgamma$ at $\theta$ by  
\begin{equation}
    \tangm_{\theta}:=\spn[\bgamma'(\theta), \dots, \bgamma^{(m)}(\theta)].
\end{equation}
Let $\pim_{\theta}$ denote the orthogonal projection to $\tangm_{\theta}$. 

\begin{theorem}\label{221013theorem1_1}
Let $n\ge 2, 1\le m\le n$ and $\bgamma: [-1, 1]\to \R^n$ be a non-degenerate curve. Let $E\subset [0, 1]^n$ be a Borel measurable set. Then 
\begin{equation}
    \dim(
    \pim_{\theta}(E)
    )= \min\{\dim(E), m\},
\end{equation}
for almost every $\theta\in [-1, 1]$. Here $\dim$ refers to the Hausdorff dimension. 
\end{theorem}

The role that $\gamma$ plays in Conjecture \ref{conj: restricted proj} is replaced by that of the derivative $\bgamma'$ in Theorem \ref{221013theorem1_1}. For instance, if we take 
\begin{equation}
\bgamma(\theta)=(\theta, \theta^2/2!, \theta^3/3!),
\end{equation}
then 
\begin{equation}\label{240120example1}
\bgamma'(\theta)=(1, \theta, \theta^2/2).
\end{equation}
Recall that $\gamma$ from Conjecture \ref{conj: restricted proj} takes values in $\mathbb{S}^2$. We normalize \eqref{240120example1}, and direct computations show that it satisfies  \eqref{240120non_degenerate}.

The study of restricted projections in general dimensions $\R^n$, similar to the ones in Theorem \ref{221013theorem1_1}, already appeared in 
J\"arvenp\"a\"a, J\"arvenp\"a\"a, Ledrappier and Leikas \cite{JJLL08}, J\"arvenp\"a\"a, J\"arvenp\"a\"a and Keleti \cite{JJK14}. The notion of ``non-degenerate curves" in the current paper is a lot stronger than the ones in \cite{JJLL08} and \cite{JJK14}, and therefore covers a smaller family of curves. On the other hand, the results obtained in Theorem \ref{221013theorem1_1} is also relatively stronger than those in \cite{JJLL08} and \cite{JJK14}.

Our approach follows the framework in \cite{GGG22}, where Fourier analysis methods were used to study incidences between points and fat $k$--planes. The main difficulty in the current paper lies in the decomposition of frequencies and the application of decoupling inequalities.  Our frequency decomposition can be viewed as some generalization of the frequency decomposition in \cite{GWZ}. It is worth mentioning that to obtain Theorem \ref{221013theorem1_1}, we do not need the sharp $L^p$ ranges of decoupling inequalities as in \cite{BDG16} or \cite{BGH21}; every sufficiently large $p$, depending on $n$ and $m$, will work, and for convenience we will work with $p=n(n+1)$ (see the paragraph above \eqref{221014e2_10}). \\


\noindent {\bf Notation.} 
\begin{enumerate}
    
    \item For two positive real numbers $R_1$ and $R_2$, we say that $R_1\lesssim R_2$ if there exists a large constant $C$, depending on relevant parameters, such that $R_1\le C R_2$; we say that $R_1\ll R_2$ if $R_1\le R_2/C$. 
    
    \item Let $\mu$ be a compactly supported Borel measure on $\R^n$. Take $\alpha>0$. Define 
\begin{equation}\label{240120edefi}
    c_{\alpha}(\mu):=\sup_{\bfx\in \R^n, r>0}\frac{\mu(B(\bfx, r))}{r^{\alpha}},
\end{equation}
where $B(x, r)$ is the ball of radius $r$ centered at $x\in \R^n$.

\item For a set $E\subset \R^n$ and $r>0$, we let $r\cdot E$ denote the set $\{r\bfx: \bfx\in E\}$. Let $T\subset \R^n$ be a rectangular box, we let $rT$ be the $r$-dilation of $T$ with respect to the center of $T$. 
\item We make the convention that $k!=\infty$ for every $k\in \Z$, $k<0$. Under this convention, the monomial $t^j/k!$ is always constantly zero whenever $k<0.$
\item We use $\bxi, \bfeta$ to denote frequency variables; they are all $n$-tuples of real numbers. 
\end{enumerate}

\medskip

\noindent {\bf Acknowledgment.} S. Guo is partly supported by NSF-1800274 and NSF-2044828. H. Wang is
supported by NSF Grant DMS-2055544. The authors would like to thank Terry Harris, Changkeun Oh, Dominique Maldague and Larry Guth for inspiring discussions. They would also like to thank the referee for carefully reading the paper, for various suggestions on the exposition of the paper, and for catching several very important typos.



\section{Idea of the proof}

We will prove Theorem \ref{221013theorem1_1} for the case 
\begin{equation}
    \bgamma(\theta)=(\frac{\theta}{1!}, \frac{\theta^2}{2!}, \dots, \frac{\theta^n}{n!}).
\end{equation}
The proof for the general case is essentially the same; the main tool we will be using, the tool of Fourier decoupling inequalities, is stable under perturbations.  \\

For $\theta\in [-1, 1]$, let $\T_{\theta}$ denote a collection of disjoint rectangular boxes of dimensions 
\begin{equation}
    \underbrace{\delta\times \dots\delta}_{m\ \mathrm{ copies}} \times \underbrace{1\times \dots1}_{n-m\ \mathrm{ copies}}
\end{equation}
covering $\R^n$, and the short sides of these rectangular boxes are parallel to $\tangm_{\theta}$. For $T\in \T_{\theta}$, we define its dual box $\tau=\tau(T)$ to be the rectangular box of dimensions 
\begin{equation}
    \underbrace{\delta^{-1}\times \dots\delta^{-1}}_{m\ \mathrm{ copies}} \times \underbrace{1\times \dots1}_{n-m\ \mathrm{ copies}}
\end{equation}
centered at the origin, whose long sides are parallel to $\tangm_{\theta}$. Sometimes we write $\tau=\tau_{\theta}$. 

For $\delta>0$, we say that $\Lambda_{\delta}\subset [-1, 1]$ is a $\delta$-net of $[-1, 1]$ if 
\begin{equation}
    \delta\le |\theta_1-\theta_2|, \ \ \forall \theta_1, \theta_2\in \Lambda_{\delta}, \theta_1\neq \theta_2,
\end{equation}
and for every $\theta\in [-1, 1]$, we can find $\theta'\in \Lambda_{\delta}$ such that $\delta\le |\theta-\theta'|<2\delta$. 

\begin{theorem}\label{221013theorem2_1}
Let $\delta>0$ and $\Lambda_{\delta}$ be a $\delta$-net of $[-1, 1]$. Given $\epsilon>0$ and $\alpha\in (0, m)$, let $\mu$ be a finite non-zero Borel measure supported in the unit ball in $\R^n$ with $c_{\alpha}(\mu)<\infty$. Take $\W_{\theta}\subset \T_{\theta}$ and denote $\W:=\cup \W_{\theta}$. Suppose that 
\begin{equation}\label{221014e2_4}
    \sum_{T\in \W} \chi_T(\bfx)\gtrsim \delta^{\epsilon-1}, \ \ \forall \bfx\in \supp(\mu). 
\end{equation}
Then 
\begin{equation}
    |\W|\ge C_{\epsilon, n, \alpha}\cdot \mu(\R^n) c_{\alpha}(\mu)^{-1} \delta^{-1-\alpha} \delta^{O(\sqrt{\epsilon})}.
\end{equation}
Here it is important that the constant $C_{\epsilon, n, \alpha}$ does not depend on $\delta$. The term $O(\sqrt{\epsilon})$ can be taken to be $10^{10n} \sqrt{\epsilon}$. 
\end{theorem}

\begin{proof}[Proof of Theorem \ref{221013theorem1_1} by assuming Theorem \ref{221013theorem2_1}.]

Let $\alpha\in (0, n)$ and $\alpha^*\in (0, \min\{m, \alpha\})$. Let $\mu$ be a finite non-zero Borel measure supported on the unit ball in $\R^n$ with $c_{\alpha}(\mu)\le 1$. Let $\delta>0$ be a small number. Let $\Lambda_{\delta}$ be a $\delta$-net. For each $\theta\in \Lambda_{\delta}$, let $\D_{\theta}$ be a disjoint collection of at most $\mu(\R^n)\delta^{-\alpha^*}$ balls of radius $\delta$ in $\pim_{\theta}(\R^n)$, and let $\mu_{\theta}:=\big(\Pi_\theta^{(m)}\big)_*\mu$ be the pushforward measure on $\pim_{\theta}(\R^n)$. Then there exists $\varepsilon_0>0$, depending only on $\alpha$ and $\alpha^*$, such that 
\begin{equation}\label{220515e2_3}
    \delta \sum_{\theta\in \Lambda_{\delta}}  \mu_{\theta}\pnorm{\bigcup_{D\in \D_{\theta}}D} \lesim_{\alpha, \alpha^*} \mu(\R^n) \delta^{\varepsilon_0}, 
\end{equation}
where the implicit constant depends on $\alpha, \alpha^*$, but not on $\delta$. The proof of \eqref{220515e2_3} is the same as that of Corollary 2 in \cite{GGG22}, and is therefore left out. \\

Assume without loss of generality that $m\ge \dim(E) >0$. Let $\epsilon' >0$ be small and let $\alpha = \dim(E) - \epsilon'$. Using Frostman's lemma (see for instance \cite[page 112]{Mat95}), we can find  a nonzero, finite Borel measure $\mu$ on $E$ with $c_{\alpha}(\mu) \leq 1$. Suppose that $\Theta \subseteq [0, 1]$ is such that
\[ \dim (\pim_{\theta}(\supp(\mu))) \leq s := \alpha - 2\epsilon', \]
for every $\theta \in \Theta$. Let $\epsilon''>0$ be small. For each $\theta \in \Theta$, we apply \cite[Lemma 4.6]{Mat95} and find 
\begin{equation}
    \left\{ B\left(\pim_{\theta}(\bfx_j(\theta)), r_j(\theta) \right) \right\}_{j=1}^{\infty}, \ \ \bfx_j(\theta)\in \R^n,
\end{equation}
a covering of $\pim_{\theta} (\supp (\mu))$ by discs of dyadic radii smaller than $\epsilon''$, such that 
\begin{equation} \label{notubes} \sum_j r_j(\theta)^{s+\epsilon'} < 1. \end{equation}
For each $\theta \in \Theta$ and each $k \geq \left\lvert \log_2 \epsilon''\right\rvert$, let 
\begin{equation}
    \D_{\theta, k}:=\{B\left(\pim_{\theta}(\bfx_j(\theta)), r_j(\theta) \right): r_j(\theta)=2^{-k}\},
\end{equation}
and 
\begin{equation}
    \bfD_{\theta, k}:=\bigcup_{D\in \D_{\theta, k}} D. 
\end{equation}
Then for each $\theta \in \Theta$, it holds that 
\[\sum_{k \geq \left\lvert \log_2 \epsilon''\right\rvert} \mu_{\theta}(\bfD_{\theta, k})\gtrsim_{\mu} 1. \]
By Vitali's covering lemma, for each $\theta \in \Theta$ and each $k$, we can find a disjoint subcollection $\D'_{\theta, k}\subset \D_{\theta, k}$ such that 
\[ 1 \lesssim_{\mu} \sum_{k \geq \left\lvert \log_2 \epsilon''\right\rvert} \mu_{\theta}(\bfD''_{\theta, k}), \]
where 
\begin{equation}
    \bfD''_{\theta, k}:=\bigcup_{D'\in \D'_{\theta, k}} (10^n D'). 
\end{equation}
As a consequence, we obtain 
\[ \mathcal{H}^1(\Theta) \lesim_{\mu} \sum_{k \geq \left\lvert \log_2 \epsilon''\right\rvert} \int_{\Theta}  \mu_{\theta}(\bfD''_{\theta, k}) \, d\theta. \]
By \eqref{notubes}, for each $\theta$ and $k$ the set $\D'_{\theta, k}$ is the union of at most $2^{k(s+\epsilon')}$ disjoint discs of radius $2^{-k}$. By \eqref{220515e2_3} and the triangle inequality, we can find $\epsilon_0>0$ independent of $k$ and $\epsilon''$, such that
\begin{align*} \mathcal{H}^1(\Theta) 
&\lesssim   \sum_{k \geq \left\lvert \log_2 \epsilon''\right\rvert} 2^{-k\epsilon_0}. \end{align*} 
Letting $\epsilon'' \to 0$ gives $\mathcal{H}^1(\Theta) = 0$. Hence 
\[ \dim (\pim_{\theta}(E)) \geq  \dim (\pim_{\theta}(\supp(\mu))) \geq \alpha - 2\epsilon'\]  
for almost every $\theta \in [0, 1]$. The theorem follows by letting $\epsilon' \to 0$ along a countable sequence.\end{proof}

In the rest of the paper we will prove Theorem \ref{221013theorem2_1}. The proof is via an induction on $\epsilon$. The base case $\epsilon\simeq 1$ of the induction is trivial. To prove Theorem \ref{221013theorem2_1} for $\epsilon$, let us assume that we have proven it for $\widetilde{\epsilon}=\frac{\epsilon}{1-\sqrt{\epsilon}}$. \\

For each $T\in \T_{\theta}$, let $\phi_T: \R^n\to \R$ be an $L^{\infty}$ normalized non-negative smooth bump function associated to $T$ with $\widehat{\phi}_T$ supported on $\tau_{\theta}$ and 
\begin{equation}
    \phi_T(\bfx)\gtrsim 1,  \ \ \forall \bfx\in T. 
\end{equation}
Note that by the assumption \eqref{221014e2_4}, we have 
\begin{equation}
    \sum_{T\in \W} \phi_T(\bfx)\gtrsim \delta^{-1+\epsilon}, \ \ \forall \bfx\in \supp(\mu). 
\end{equation}
Let $\psi_{0, \delta}: \R^n\to \R$ be a smooth bump function that is supported on $B(0, 2\delta^{-1+\sqrt{\epsilon}})\subset \R^n$ and equals 1 on $B(0, \delta^{-1+\sqrt{\epsilon}})\subset \R^n$. We write 
\begin{equation}
    \widehat{\phi}_T=\widehat{\phi}_T\cdot \psi_{0, \delta}+\widehat{\phi}_T\cdot (1-\psi_{0, \delta}). 
\end{equation}
The former term is called the low frequency part of $\phi_T$, and the latter term is called the high frequency part of $\phi_T$. We can find a Borel set $F$ satisfying $\mu(F)\gtrsim \mu(\R^n)$ such that either 
\begin{equation}\label{221030e2_16}
    \delta^{-1+\epsilon}\lesim \anorm{
    \sum_{T\in \W} \pnorm{\phi_T* \widecheck{\psi}_{0, \delta}}(\bfx)
    }, \ \ \forall \bfx\in F,
\end{equation}
or 
\begin{equation}\label{221030e2_17}
    \delta^{-1+\epsilon}\lesim \anorm{
    \sum_{T\in \W} \pnorm{\phi_T* 
    (1-\psi_{0, \delta})^{\lor}
    }(\bfx)
    }, \ \ \forall \bfx\in F,
\end{equation}

Let us assume we are in the former case. For each $T\in \W$, the function $\phi_T* \widecheck{\psi_{0, \delta}}$ is essentially supported on a rectangular box of dimensions 
\begin{equation}
    \underbrace{(\rho^{-1}\delta)\times \dots\times (\rho^{-1}\delta)}_{m \ \mathrm{copies}}\times \underbrace{1\times \dots\times 1}_{(n-m) \mathrm{ copies}}
\end{equation}
with $\rho:=\delta^{\sqrt{\epsilon}}$; we use $\widetilde{T}$ to denote this rectangular box. We define a relation. Take $\widetilde{T}_1 $ and $\widetilde{T}_2 $. We say that $\widetilde{T}_1 \sim \widetilde{T}_2 $ if $\widetilde{T}_2 \subset 10^n \widetilde{T}_1 $ or the other way around. We let $\widetilde{W}$ be a collection of the enlarged boxes $\widetilde{T}$ satisfying that 
\begin{equation}
    \widetilde{T}_1\not\sim \widetilde{T}_2, \ \ \forall \widetilde{T}_1, \widetilde{T}_2\in \widetilde{W}, \widetilde{T}_1\neq \widetilde{T}_2,
\end{equation}
and for every $T\in W$, there exists $T'$ with $\widetilde{T'}\in \widetilde{W}$ such that $\widetilde{T}\sim \widetilde{T'}$. Write $\widetilde{W}$ as a disjoint union 
\begin{equation}
\widetilde{W}=\wheavy\bigsqcup \wlight,
\end{equation}
where we use $\wlight$ to collect the enlarged boxes $\widetilde{T}$ that contains $\le C^{-1} \rho^{-m-1+\sqrt{\epsilon}}$ many boxes from $\W$. Here $C$ is a large universal constant that is much larger than the implicit universal constant in \eqref{221030e2_16}. Note that 
\begin{equation}
    \Norm{
    \sum_{\widetilde{T}\in \wlight}
    \sum_{T'\in \W: T'\subset \widetilde{T}}
    \phi_{T'}* \widecheck{\psi}_{\delta}
    }_{\infty}
    \le 
    \sum_{\widetilde{T}\in \wlight}
    \sum_{T'\in \W: T'\subset \widetilde{T}}
    \Norm{
    \phi_{T'}* \widecheck{\psi}_{\delta}
    }_{\infty}
    \ll (\delta^{-1}\rho) \rho^m \rho^{-m-1+\sqrt{\epsilon}}.
\end{equation}
As a consequence, we see that 
\begin{equation}
    \delta^{\epsilon-1}\lesim \anorm{
    \sum_{\widetilde{T}\in \wheavy}
    \sum_{T'\in \W: T'\subset \widetilde{T}}
    \phi_{T'}* \widecheck{\psi}_{\delta}(\bfx)
    }, \ \ \forall \bfx\in F. 
\end{equation}
Note that each $\widetilde{T}$ contains at most $\rho^{-m-1}$ many boxes from $\W$. We can conclude that 
\begin{equation}
    \sum_{\widetilde{T}\in \wheavy} \chi_{\widetilde{T}}(\bfx)\gtrsim 
    \delta^{\epsilon-1}\rho, \ \ \forall \bfx\in F. 
\end{equation}
Write 
\begin{equation}
    \delta^{-1+\epsilon}\rho=(\delta \rho^{-1})^{-1+\widetilde{\epsilon}}, \ \ \widetilde{\epsilon}=\frac{\epsilon}{1-\sqrt{\epsilon}}.
\end{equation}
The boxes in $\wheavy$ satisfy the induction hypothesis at the scale $(\delta \rho^{-1})$ with the new parameter $\widetilde{\epsilon}$. Hence 
\begin{equation}
    |\W|\gtrsim |\wheavy| \rho^{-m-1+\sqrt{\epsilon}}
    \gtrsim 
    \mu(\R^n) c_{\alpha}(\mu)^{-1} (\delta^{-1}\rho)^{\alpha+1-10^{10n}\sqrt{\widetilde{\epsilon}}} \rho^{-m-1+\sqrt{\epsilon}}.
\end{equation}
Elementary computations yield that 
\begin{equation}
    |\W|\gtrsim \mu(\R^n) c_{\alpha}(\mu)^{-1} 
    \delta^{-(\alpha+1-10^{10n}\sqrt{\epsilon})}
    \rho^{-m+\alpha}
    \delta^{10^{10n}\frac{3\epsilon}{4}-10^{10n}\epsilon+\epsilon}.
\end{equation}
Since $\alpha\le m$, the induction closes. \\

We turn to the case \eqref{221030e2_17}. This case will be handled directly, without using induction. Let $p:=n(n+1)$. We raise both sides to the $p$-th power, integrate with respect to $d\mu$, replace $d\mu$ by the Lebesgue measure and obtain 
\begin{equation}\label{221014e2_10}
    \mu(\R^n) c_{\alpha}(\mu)^{-1} \delta^{-p+p\epsilon}\delta^{n-\alpha} \lesim \int \anorm{
    \sum_T \phi_T*(1-\psi_{0, \delta})^{\lor}
    }^p.
\end{equation}
Here we used the definition of $c_{\alpha}(\mu)$ in \eqref{240120edefi} and the fact that the integrated function is essentially constant on $\delta$-balls. By decomposing the support of $1-\psi_{0, \delta}$ into dyadic annuli and applying the triangle inequality, we can without loss of generality assume that 
\begin{equation}\label{221014e2_10zz}
    \mu(\R^n) c_{\alpha}(\mu)^{-1} |\log \delta|^{-p} \delta^{-p+p\epsilon}\delta^{n-\alpha} \lesim  \int \anorm{
    \sum_T \phi_T*\widecheck{\psi_{1, \delta}}
    }^p,
\end{equation}
where $\psi_{1, \delta}$ is an $L^{\infty}$-normalized bump function supported on $\{\bxi\in \R^n: |\bxi|\simeq \delta^{-1}\}$.  
To simplify further notation, we will rescale the right hand side so that frequency variables take values on a compact set $B(0, C)$ for some constant $C\simeq 1$. More precisely, 
\begin{equation}\label{221014e2_11}
    \begin{split}
        \eqref{221014e2_10zz} & \lesim \int \anorm{
        \sum_T \int \widehat{\phi}_T(\bxi)\psi_{1, \delta}(\bxi)e^{i\bfx\cdot \bxi}d\bxi
        }^p d\bfx\\
        & \lesim  \delta^{-pn+n} \int \anorm{
        \sum_T \int \widehat{\phi}_T(\delta^{-1} \bxi)\psi_{1,  \delta}(\delta^{-1}\bxi)e^{i\bfx\cdot \bxi}d\bxi
        }^p d\bfx
    \end{split}
\end{equation}
Define
\begin{equation}\label{221014e2_12}
    \phicirc_T(\bfx):=\phi_T(\delta \bfx), \ \ \psi_{1}(\bxi):=\psi_{1, \delta}(\delta^{-1}\bxi),
\end{equation}
and
\begin{equation}
    \widehat{\phiplus_T}(\bxi):=
    \widehat{\phicirc_T}(\bxi) \psi_{1}(\bxi).
\end{equation}
Under the new notation, \eqref{221014e2_11} can be written as 
\begin{equation}
    \eqref{221014e2_10zz}\lesim \delta^n \int 
    \anorm{
    \sum_T \int 
    \widehat{\phiplus_T}(\bxi)
    e^{i\bfx\cdot \bxi}d\bxi
    }^p d\bfx.
\end{equation}
It therefore remains to prove
\begin{equation}\label{221027e2_16ppp}
    \int 
    \anorm{
    \sum_T \int 
    \widehat{\phiplus_T}(\bxi)
    e^{i\bfx\cdot \bxi}d\bxi
    }^p d\bfx \lesim_{\epsilon, p, n, \alpha} \delta^{-p+1} \delta^{-\epsilon} |\W|.
\end{equation}
Similarly to \eqref{221014e2_12}, define
\begin{equation}
    \tcirc:=\delta\cdot T, \ \ \taucirc:=\delta^{-1}\cdot\tau.
\end{equation}
After the above rescaling, we see that for $\bxi\in \supp(\widehat{\phiplus_T})$ with $T\in \T_{\theta}$, it can be written as 
\begin{equation}
    \lambda_1 \bgamma'(\theta)+\dots +\lambda_n \bgamma^{(n)}_n(\theta),
\end{equation}
with 
\begin{equation}\label{221018e2_17}
    |\lambda_{\iota}|\le 1, |\lambda_{\iota'}|\le \delta, \ \ 1\le \iota\le m, m+1\le \iota'\le n.
\end{equation}
and 
\begin{equation}\label{221018e2_18}
    \sum_{\iota} |\lambda_{\iota}|\simeq  1.
\end{equation}
The support of $\widehat{\phi}^+_T$ will be denoted by $\Omega^+_T$. In the next few sections, we will decompose $\Omega^+_T$ into smaller pieces and apply decoupling inequalities. Roughly speaking, the decomposition is guided by the scales at which decoupling inequalities in \cite{BDG16} and \cite{BGH21} can be applied. For instance, the first step of the decomposition (see \eqref{240120decomposition} below) is to decide the largest $\iota\in \{1, 2, \dots, m\}$ in \eqref{221018e2_18} satisfying $|\lambda_{\iota}|\simeq 1$. We use $m_1$ to denote this largest $\iota$ (see \eqref{221112e3_1} below). Once $m_1$ is decided, we will not decompose the direction corresponding to $\lambda_{m_1}$. Moreover, it is key to observe here that all the directions corresponding to $\lambda_{\iota''}$ with $\iota''< m_1$ do not need to be decomposed either. \footnote{Reflected in the definition in \eqref{221112e3_1}, what we do is to put all frequencies $|\lambda_{\iota''}|\lesim1 $ together, whenever $\iota''< m_1$.} The second step of the decomposition is guided by the scales that appear in decoupling inequalities (see \eqref{220901e3_4}). After these two steps, we will apply decoupling inequalities in \cite{BDG16} and \cite{BGH21} (see Claim \ref{221027claim3_1}). For each decoupled frequency regions, we will repeat the whole process again.

\section{Frequency decomposition and decoupling inequalities: Step One.}

For $m_1\le m$, denote 
\begin{equation}\label{221112e3_1}
    \Omega^+_{T,  m_1}:=\{\sum_{\iota=1}^n \lambda_{\iota} \bgamma^{(\iota)}(\theta): T\in \W_{\theta}, |\lambda_{\iota''}|\lesim 1, \forall \iota''<m_1,   |\lambda_{m_1}|\simeq 1, |\lambda_{\iota'}|\ll 1, \ \forall m_1<\iota'\le m\}.
\end{equation}
To slightly simplify future notation, we make the convention that whenever we have a sum as in \eqref{221112e3_1} that runs over all $1\le \iota\le n$, we will simply neglect the range of $\iota$; moreover, if $\lambda_{\iota}$ takes values in the largest possible range, that is, $|\lambda_{\iota}|\lesim 1$, then we would not specify the range of $\lambda_{\iota}$.   We have
\begin{equation}\label{240120decomposition}
    \Omega^+_T=\bigcup_{m_1\le m}
    \Omega^+_{T, m_1}.
\end{equation}
Define $\phiplus_{T, m_1}$ to be a frequency projection of $\phiplus_T$ to $\Omega^+_{T, m_1}$ so that 
\begin{equation}
    \phiplus_T=\sum_{m_1} \phiplus_{T, m_1}.
\end{equation}
We further decompose $\Omega^+_{T, m_1}$. We discuss the case $m_1>m$ and the case $m_1=m$ separately. 
For $m_1>m$ and positive dyadic numbers $s_1$ with 
\begin{equation}
    1\gg s_1\ge  \delta^{\frac{1}{n-m_1}},
\end{equation}
denote 
\begin{equation}\label{220901e3_4}
    \Omega^+_{T, m_1, s_1}:=\{\sum_{\iota} \lambda_{\iota} \bgamma^{(\iota)}(\theta)\in \Omega^+_{T, m_1}: s_1^{-1} |\lambda_{m_1+1}|+\dots+ s_1^{-(m-m_1)}|\lambda_m|\simeq 1\},
\end{equation}
with a modification that $\simeq$ is replaced by $\lesim$ when $s_1=\delta^{\frac{1}{n-m_1}}$. If we are in the case $m_1=m$, then $s_1$ takes only one value $\delta^{\frac{1}{n-m_1}}$, and we denote 
\begin{equation}
\Omega^+_{T, m_1, s_1}:=
\Omega^+_{T, m_1}.
\end{equation}
In other words, in the case $m_1=m$, we do not decompose $\Omega^+_{T, m_1}$, and will apply decoupling inequalities directly (Claim \ref{221027claim3_1} below);  the reason of introducing $s_1$ in this case is that $s_1$ will tell us at what scale decoupling inequalities will be applied. \\

By the decomposition in \eqref{220901e3_4}, we  have 
\begin{equation}
    \Omega^+_{T, m_1}=\bigcup_{s_1} \Omega^+_{T, m_1, s_1}. 
\end{equation}
We define $\phiplus_{T, m_1, s_1}$ to be a frequency projection of $\phiplus_{T, m_1}$ to the region $\Omega^+_{T, m_1, s_1}$ so that \begin{equation}
    \phiplus_{T, m_1}=\sum_{s_1} \phiplus_{T, m_1, s_1}. 
\end{equation}
\begin{claim}\label{221027claim3_1}
Let $\W_{\theta}\subset \T_{\theta}$. We have
\begin{equation}\label{221017e3_8}
    \begin{split}
        & \Norm{
        \sum_{\theta\in \Lambda_{\delta}} 
        \sum_{T\in \W_{\theta}}
        \phiplus_{T, m_1, s_1}
        }_p \lessapprox \pnorm{\frac{1}{s_1}}^{1-
        \frac{\mathfrak{D}_{n-m_1}}{p}
        } 
        \pnorm{
        \sum_{\ell(\Theta_1)=s_1}
        \Norm{
        \sum_{\theta\in \Lambda_{\delta}\cap \Theta_1}
        \sum_{T\in \W_{\theta}}
        \phiplus_{T, m_1, s_1}
        }_p^p
        }^{1/p}.
    \end{split}
\end{equation}
Here and below, $\lessapprox$ means we are allowed to lose $\delta^{-\varepsilon}$ and $\varepsilon$ can be chosen to be arbitrarily small, and 
\begin{equation}
    \mathfrak{D}_{\iota}:=\frac{\iota(\iota+1)+2}{2}.
\end{equation}
\end{claim}
\begin{proof}[Proof of Claim \ref{221027claim3_1}] The proof is essentially contained in \cite{BGH21}, and we only give a sketch of the proof. The proof uses the decoupling inequality in \cite{BDG16} and the bootstrapping argument in \cite{PS07}. To simplify notation, let us write 
\begin{equation}
    F_{\Theta_1}:=\sum_{\theta\in \Lambda_{\delta}\cap \Theta_1}
        \sum_{T\in \W_{\theta}}
        \phiplus_{T, m_1, s_1}.
\end{equation}
Note that, because $s_1\ge \delta^{\frac{1}{n-m_1}}$, the support of $\widehat{F_{\Theta_1}}$ is contained in 
\begin{equation}
    \{\sum_{\iota} \lambda_{\iota} \bgamma^{(\iota)}(\theta): \theta\in \Theta_1, |\lambda_{m_1}|\simeq 1, |\lambda_{m_1+1}|\lesim s_1, \dots, |\lambda_n|\lesim s_1^{n-m_1}\}=:\Omega_{\Theta_1}.
\end{equation}
The region $\Omega_{\theta_1}$ is essentially a rectangular box of dimensions 
\begin{equation}
    1\times \dots\times 1\times s_1\times \dots\times s_1^{n-m_1}.
\end{equation}
For an interval $\Theta$ with $|\Theta|\ge s_1$, we define 
\begin{equation}
    F_{\Theta}:=\sum_{\Theta_1\subset \Theta} F_{\Theta_1}. 
\end{equation}
Let $\varepsilon>0$ be a small number. By the triangle inequality, we can cut the region $\Omega_{\Theta_1}$ into smaller regions of dimensions 
\begin{equation}\label{221110e3_14}
    s_1^{\varepsilon}\times \dots\times s_1^{\varepsilon} \times s_1\times \dots\times s_1^{n-m_1},
\end{equation}
and restrict $\Theta_1$ to a small interval of length $s_1^{\varepsilon}$, say $[0, s_1^{\varepsilon}]$. Without loss of generality, let us assume that $\lambda_{m_1}\in [1, 1+s_1^{\varepsilon}]$. To avoid introducing new notation, we still use $\Omega_{\Theta_1}$ to denote a region of the scale \eqref{221110e3_14} with $\lambda_{m_1}\in [1, 1+s_1^{\varepsilon}]$, and still use the notation $F_{\Theta_1}$ and $F_{\Theta}$. We need to prove 
\begin{equation}\label{201110e3_15}
    \norm{
    F_{\mathfrak{I}_1}
    }_p 
    \lessapprox
    \pnorm{
    \frac{1}{s_1}
    }^{1-\frac{\mathfrak{D}_{n-m_1}}{p}}
    \pnorm{
    \sum_{\Theta_1\subset \mathfrak{I}_1}
    \norm{
    F_{\Theta_1}
    }_p^p
    }^{1/p},
\end{equation}
where $\mathfrak{I}_1:=[0, s_1^{\varepsilon}]$. To prove \eqref{201110e3_15}, first observe that the Fourier transform of $F_{\mathfrak{I}_1}$ is ``close" to the cylinder 
\begin{equation}\label{221110e3_16}
    \bigg\{\lambda'_1 \vec{e}_1+\dots+\lambda'_{m_1} \vec{e}_{m_1}+\Big(0, \dots, 0, \frac{\theta}{1!}, \dots, \frac{\theta^{n-m_1}}{(n-m_1)!}\Big):
    \begin{array}{l}
         |\lambda'_1-a_1|, \dots, |\lambda'_{m_1-1}-a_{m_1-1}|\le s_1^\e,  \\
         \lambda'_{m_1}\in [1, 1+s_1^{\varepsilon}], \theta\in \mathfrak{I}_1 
    \end{array}\bigg\}.
\end{equation}
Here $\vec{e}_j$ are coordinate vectors, and $a_1, \dots, a_{m_1-1}$ are numbers lying in $[-1,1]$. Let us be more precise. A point from $\Omega_{\Theta_1}$ can be written as
\begin{equation}
    \sum_{\iota=1}^{m_1-1} \lambda_{\iota} \bgamma^{(\iota)}(\theta)+
    \lambda_{m_1} \bgamma^{(m_1)}(\theta)
    +\sum_{\iota=m_1+1}^{n} \lambda_{\iota} \bgamma^{(\iota)}(\theta)=:\bfQ.
\end{equation}
Let us compute the distance of this point to the following point on the cylinder \eqref{221110e3_16}: 
\begin{equation}
    \bfQ':=
    \pnorm{
    \lambda_1, \lambda_2+\lambda_1\theta, \dots, \lambda_{m_1}+\dots+\lambda_1 \frac{\theta^{m_1-1}}{(m_1-1)!}, 
    \lambda_{m_1}\theta, \frac{(\lambda_{m_1}\theta)^2}{2!}, \dots, 
    \frac{(\lambda_{m_1}\theta)^{n-m_1}}{(n-m_1)!}
    }.
\end{equation}
Write 
\begin{equation}
    \bfQ=(Q_1, \dots, Q_n), \ \ \bfQ'=(Q'_1, \dots, Q'_n).
\end{equation}
We have that 
\begin{equation}
    |Q_{\iota}-Q'_{\iota}|=0, \forall 1\le \iota\le m_1,
\end{equation}
and 
\begin{equation}
    |Q_{\iota}-Q'_{\iota}|\lesim s_1^{\varepsilon(\iota-m_1+1)}, \forall \iota\ge m_1+1.
\end{equation}
By an anisotropic rescaling in the frequency variables of $F_{\mathfrak{I}_1}$ and applying the decoupling inequality\footnote{Decoupling inequalities in \cite{BDG16} are stated for the moment curve, but the proof there works for general non-degenerate curves. One can also use the bootstrapping argument in \cite{PS07}, as is done in Lemma 3.6 in \cite{GLYZ21}.} in \cite{BDG16} to the cylinder \eqref{221110e3_16}, we obtain 
\begin{equation}
    \norm{F_{\mathfrak{I}_1}
    }_p \lesssim 
    \pnorm{\frac{1}{\delta}}^{\varepsilon^2}
    \pnorm{
    \frac{1}{s_1^{\frac{\varepsilon}{n-m_1}}}
    }^{1-
    \frac{\mathfrak{D}_{n-m_1}}{p}
    }
    \pnorm{
    \sum_{\mathfrak{I}_2\subset \mathfrak{I}_2, |\mathfrak{I}_2|=|\mathfrak{I}_1| s_1^{\frac{\varepsilon}{n-m_1}}}
    \norm{F_{\mathfrak{I}_2}}_p^p
    }^{\frac{1}{p}}.
\end{equation}
Note that this decoupling inequality can be iterated, in an essentially the same way as done above, for each resulting term $\norm{F_{\mathfrak{I}_2}}_p$. After $(n-m_1)/\varepsilon$ many steps, we will eventually obtain the desired estimate. 
\end{proof}

After applying decoupling, we apply rescaling to 
\begin{equation}
    \sum_{\theta\in \Lambda_{\delta}\cap \Theta_1}
        \sum_{T\in \W_{\theta}}
        \phiplus_{T, m_1, s_1}.
\end{equation}
Without loss of generality, we assume that we are working with $\Theta_1=[0, s_1]$. 
On the frequency side, we do the scaling
\begin{equation}\label{220901e3_13}
    \xi_1\mapsto \eta_1, \ \dots, \ \xi_{m_1}\mapsto \eta_{m_1}, \ \xi_{m_1+1}\mapsto s_1 \eta_{m_1+1}, \dots, \xi_n\mapsto s_1^{n-m_1}\eta_n.
\end{equation}
For each $T\in \W_{\theta}, \theta\in \Lambda_{\delta}\cap \Theta_1$, we study the support of the function 
\begin{equation}
    \widehat{\phi}^{+}_{T_{s_1}}(\bfeta):=
    \widehat{\phi}^{+}_{T, m_1, s_1}(\eta_1, \dots, \eta_{m_1}, s_1 \eta_{m_1+1}, \dots, s_1^{n-m_1} \eta_n).
\end{equation}
Note that a frequency point under consideration can be written as 
\begin{equation}
    \begin{split}
        &\bxi=\lambda_1 \bgamma'(\theta)+\dots+\lambda_{m_1} \bgamma^{(m_1)}(\theta)+
        \lambda_{m_1+1} \bgamma^{(m_1+1)}(\theta)
        +\dots \lambda_n \bgamma^{(n)}(\theta),
    \end{split}
\end{equation}
with 
\begin{equation}
    \frac{|\lambda_{m_1+1}|}{s_1}+\dots+\frac{|\lambda_m|}{s_1^{m-m_1}}\simeq 1,
\end{equation}
and $\simeq $ is replaced by $\lesim $ when $s_1=\delta^{\frac{1}{n-m_1}}$. Under the above change of variables, it becomes 
\begin{equation}\label{220901e3_14}
    \begin{split}
    \pnorm{
    & 
    \sum_{\iota} \lambda_{\iota} \frac{\theta^{1-\iota}}{(1-\iota)!},
    \dots, 
    \sum_{\iota} \lambda_{\iota} \frac{\theta^{m_1-\iota}}{(m_1-\iota)!}, \\
    & 
    s_1^{-1}
    \sum_{\iota} \lambda_{\iota} \frac{\theta^{m_1+1-\iota}}{(m_1+1-\iota)!}, 
    \dots, 
    s_1^{-(n-m_1)}
    \sum_{\iota} \lambda_{\iota} \frac{\theta^{n-\iota}}{(n-\iota)!}
    }=:\bfeta
    \end{split}
\end{equation}
We rename variables 
\begin{equation}
    \frac{\lambda_{m_1+1}}{s_1}\mapsto \lambda_{m_1+1}, \dots, \frac{\lambda_n}{s_1^{n-m_1}}\mapsto \lambda_n,
\end{equation}
and write 
\begin{equation}\label{220901e3_17}
    \begin{split}
    \bfeta=\pnorm{
    & 
    \sum_{\iota} \lambda_{\iota} \frac{\theta^{1-\iota}}{(1-\iota)!},
    \dots, 
    \sum_{\iota} \lambda_{\iota} \frac{\theta^{m_1-\iota}}{(m_1-\iota)!}, \\
    & 
    s_1^{-1}
    \sum_{\iota=1}^{m_1} \lambda_{\iota} \frac{\theta^{m_1+1-\iota}}{(m_1+1-\iota)!}+\sum_{\iota=m_1+1}^n s_1^{\iota-m_1-1} \lambda_{\iota} \frac{\theta^{m_1+1-\iota}}{(m_1+1-\iota)!}, \\
    & \dots, s_1^{-(n-m_1)}
    \sum_{\iota=1}^{m_1} \lambda_{\iota} \frac{\theta^{n-\iota}}{(n-\iota)!}+\sum_{\iota=m_1+1}^n s_1^{\iota-n} \lambda_{\iota} \frac{\theta^{n-\iota}}{(n-\iota)!}}
    \end{split}
\end{equation}
In the end, we rename $\theta\mapsto s_1 \theta$, and write 
\begin{equation}\label{220901e3_18}
    \begin{split}
    \bfeta=\lambda_1 \frac{\bgamma'_{m_1, s_1}(\theta)}{s_1^{1-1}}+\dots+\lambda_{m_1}\frac{\bgamma_{m_1, s_1}^{(m_1)}(\theta)}{s_1^{m_1-1}} +\lambda_{m_1+1}\frac{\bgamma_{m_1, s_1}^{(m_1+1)}(\theta)}{s_1^{m_1-1}}+\dots+
    \lambda_n \frac{\bgamma_{m_1, s_1}^{(n)}(\theta)}{s_1^{m_1-1}}
    \end{split}
\end{equation}
where
\begin{equation}
    \bgamma_{m_1, s_1}(\theta):=
    (\mathcal{R}_{m_1, s_1}\circ \bgamma)(\theta),
\end{equation}
\begin{equation}
    (\mathcal{R}_{m_1, s_1}\circ \bgamma)(\theta):=\pnorm{s_1^{1-1}
    \frac{\theta}{1!}, \dots, s_1^{k-1}\frac{\theta^k}{k!},\dots, s_1^{m_1-1}\frac{\theta^{m_1}}{m_1!}, 
    s_1^{m_1-1}\frac{\theta^{m_1+1}}{(m_1+1)!}, \dots, s_1^{m_1-1}\frac{\theta^{n}}{n!}
    },
\end{equation}
and 
\begin{align}\label{220902e3_24}
        & |\lambda_1|, \dots, |\lambda_{m_1-1}| \lesim 1, \ \  |\lambda_{m_1}|\simeq 1, \\ & |\lambda_{m_1+1}|+\dots+|\lambda_m|\simeq 1,\label{221024e3_22aa}\\
        & |\lambda_{m+1}|\lesim s_1^{-(m+1-m_1)}\delta, \dots, |\lambda_n|\lesim s_1^{-(n-m_1)}\delta, 
\end{align}
with $\simeq$ in \eqref{221024e3_22aa} replaced by $\lesim$ when $s_1^{n-m_1}=\delta$. Moreover, if for $\blambda=(\lambda_1, \dots, \lambda_n)$, we define 
\begin{equation}
    \mathcal{D}_{m_1, s_1}(\blambda):=\pnorm{
    \frac{\lambda_1}{s_1^{1-1}}, \dots, \frac{\lambda_{m_1}}{s_1^{m_1-1}}, \frac{\lambda_{m_1+1}}{s_1^{m_1-1}}, \dots, \frac{\lambda_n}{s^{m_1-1}}
    }, 
\end{equation}
then \eqref{220901e3_18} can be simplified to 
\begin{equation}
    \bfeta=\mathcal{D}_{m_1, s_1}(\blambda)\cdot \pnorm{
    \bgamma'_{m_1, s_1}(\theta), \dots, \bgamma_{m_1, s_1}^{(n)}(\theta)
    }.
\end{equation}
The expression \eqref{220901e3_18} looks complicated. However, the most important terms there are the ones corresponding to coefficients $\lambda_{m_1}, \dots, \lambda_n$, and they can be written as 
\begin{equation}
    \lambda_{m_1} \bgamma^{(m_1)}(\theta)+\dots+ \lambda_n \bgamma^{(n)}(\theta). 
\end{equation}
So far we have done a change of variables so that 
\begin{equation}\label{221024e3_22}
        \Norm{
        \sum_{\theta\in \Lambda_{\delta}\cap \Theta_1}
        \sum_{T\in \W_{\theta}}
        \phi^{+}_{T, m_1, s_1}
        }_p\lesim \jac(s_1)\cdot 
        \Norm{
        \sum_{\theta\in \Lambda_{s_1^{-1}\delta}}
        \sum_{
        T\in 
        \W_{\theta}(s_1; \Theta_1)
        }
        \phi^{+}_{T_{s_1}}
        }_p
\end{equation}
where $\W_{\theta}(s_1; \Theta_1)$ is what $\W_{\theta}$ is transformed to after renaming parameters $\theta\mapsto s_1\theta$, 
\begin{equation}\label{221018e3_22}
    \jac(s_1):=s_1^{(1-\frac{1}{p})(1+2+\dots+(n-m_1))},
\end{equation}
 and the support of $\widehat{\phi}^+_{T_{s_1}}$ is denoted as $\Omega^+_{T_{s_1}}$, that is, $\Omega^+_{T_{s_1}}$ collects all the $\bfeta$ satisfying \eqref{220901e3_18}. The form of the Jacobian factor in \eqref{221018e3_22} is not important, as we will revert the change of variables done above, and the Jacobian factor will be cancelled out. \\

Before we proceed to the next step, we record a few pieces of information from this step. Denote 
\begin{equation}
   \begin{split}
       &  \bfm_1=(m_1), \ \ \bfs_1=(s_1), \ \  \thick_1=(s_1^{-(m+1-m_1)}\delta, \dots, s_1^{-(n-m_1)}\delta),\\
       & 
   \end{split}
\end{equation}
and $n_1:=n$. Moreover, we record the relation $s_1^{-(n-m_1)}\le \delta^{-1}$. 

\section{Frequency decomposition and decoupling inequalities: General steps}

Let $j\ge 1$ be an integer. Suppose we have finished the $j$-th step of the frequency decomposition. So far we have collected $j$-tuples of integers $\bfm_j=(m_1, \dots, m_j)$ satisfying
\begin{equation}
    m_1\le m_2\le \dots\le m_j,
\end{equation}
$j$-tuples of dyadic numbers $\bfs_j=(s_1, \dots, s_j)$,  dimension parameters $n_j\in \{m, m+1, \dots, n\}$, an $(n-m)$-tuple
\begin{equation}
    \thick_j=\pnorm{
    \delta \prod_{j'=1}^j s_{j'}^{-(m+1-m_{j'})}, \dots, \delta \prod_{j'=1}^j s_{j'}^{-(n-m_{j'})}
    },
\end{equation}
and a relation
\begin{equation}\label{221026e4_3aa}
    \prod_{j'=1}^j s_{j'}^{-(n_j-m_{j'})}\le \delta^{-1}. 
\end{equation}
Moreover, by combining \eqref{221017e3_8}, \eqref{221024e3_22} and the triangle inequality, we obtain  \begin{equation}\label{221024e4_4}
\begin{split}
        \Norm{
    \sum_T \phi_T^+
    }_p \lessapprox & 
    \sum_{\bfm_j} \sum_{\bfs_j} 
    \pnorm{
    \prod_{j'=1}^j 
    \jac(s_{j'})
    }
    \pnorm{
    \prod_{j'=1}^j 
    \pnorm{
    \frac{1}{s_{j'}}
    }^{1-\frac{\mathfrak{D}_
    {n_{j'}-m_{j'}}}{p}}
    }\\
    & 
    \pnorm{
    \sum_{\ell(\Theta_j)=\ell_j}
    \Norm{
    \sum_{\theta\in \Lambda_{\ell_j^{-1}\delta}}
    \sum_{T\in \W_{\theta}(\bfs_j; \Theta_j)}
    \phiplus_{T_{\bfs_j}}
    }_p^p
    }^{1/p},
\end{split}
\end{equation}
where $\ell_j:=\prod_{j'=1}^j s_{j'},$ and for each $\bfeta$ in the support of $\widehat{\phiplus_{T_{\bfs_j}}}$, it can be written as 
\begin{equation}\label{221026e4_6hh}
    \bfeta=\mathcal{D}_{\bfm_j, \bfs_j}(\blambda)\cdot \pnorm{
    \bgamma'_{\bfm_j, \bfs_j}(\theta), \dots, \bgamma_{\bfm_j, \bfs_j}^{(n)}(\theta)
    },
\end{equation}
where 
\begin{equation}
    \mathcal{D}_{\bfm_j, \bfs_j}(\blambda):=\mathcal{D}_{m_j, s_j}\circ\dots\circ\mathcal{D}_{m_1, s_1}(\blambda),
\end{equation}
and 
\begin{equation}
    \bgamma_{\bfm_j, \bfs_j}:=\mathcal{R}_{m_j, s_j}\circ\dots\circ \mathcal{R}_{m_1, s_1}\circ \bgamma,
\end{equation}
with $\theta\in [0, 1]$ and $\blambda=(\lambda_1, \dots, \lambda_n)$ satisfying 
\begin{equation}\label{220902e3_24zzz}
    \begin{split}
        & |\lambda_1|, \dots, |\lambda_{m_j-1}| \lesim 1, \ \  |\lambda_{m_j}|\simeq 1, \ \ |\lambda_{m_j+1}|+\dots+|\lambda_m|\simeq 1,\\
        & |\lambda_{m+1}|\lesim 
        \delta 
        \prod_{j'=1}^j
        s_{j'}^{-(m+1-m_{j'})}, \dots, |\lambda_{n_j}|\lesim \delta\prod_{j'=1}^j s_{j'}^{-(n_j-m_{j'})}, 
    \end{split} 
\end{equation}
with $\simeq$ replaced by $\lesim$ when 
\begin{equation}\label{221026e4_10hh}
    \prod_{j'=1}^j s_{j'}^{n_j-m_{j'}}=\delta.
\end{equation}
We use $\Omega^+_{T_{\bfs_j}}$ to denote the the support of $\widehat{\phiplus_{T_{\bfs_j}}}$. 
So far we have finished collecting all the useful data from previous steps. \\

We continue our algorithm for the frequency decomposition in step $(j+1)$. To simplify our presentation, we will work with $\Theta_j=[0, \ell_j],$ where $\Theta_j$ is from the right hand side of \eqref{221024e4_4}. Define
\begin{equation}\label{221024e4_11}
    n_{j+1}:=
    \begin{cases}
    n_j & \text{ if } \prod_{j'=1}^j s_{j'}^{n_j-m_{j'}}> \delta,\\
    n_j-1 & \text{ if } \prod_{j'=1}^j s_{j'}^{n_j-m_{j'}}= \delta.
    \end{cases}
\end{equation}
If $n_{j+1}=m$, then we terminate our algorithm. If $n_{j+1}>m$, then we proceed as follows. \\

If we are in the former case in \eqref{221024e4_11}, then \eqref{220902e3_24zzz} tells us that 
\begin{equation}
    |\lambda_{m_{j}+1}|+\dots+|\lambda_m|\simeq 1. 
\end{equation}
For integers $m_{j+1}$ satisfying $m\ge m_{j+1}>m_j$, define
\begin{equation}
    \begin{split}
        \Omega^+_{T_{\bfs_j}, m_{j+1}}:=
        & \Big\{\bfeta=\mathcal{D}_{\bfm_j, \bfs_j}(\blambda)\cdot \pnorm{
    \bgamma'_{\bfm_j, \bfs_j}(\theta), \dots, \bgamma_{\bfm_j, \bfs_j}^{(n)}(\theta)
    }: \\
    & T\in \W_{\theta}(\bfs_j; \Theta_j), \bfeta\in \Omega^+_{T_{\bfs_j}}, |\lambda_{m_{j+1}}|\simeq 1, |\lambda_{\iota}|\ll 1, \forall \iota>m_{j+1}
    \Big\}
    \end{split}
\end{equation}
Next, for positive dyadic numbers $s_{j+1}$ with 
\begin{equation}\label{221024e4_15}
    1\gg s_{j+1}\ge \pnorm{
    \delta
    \prod_{j'=1}^j s_{j'}^{-(n_j-m_{j'})}
    }^{\frac{1}{n_{j+1}-m_{j+1}}},
\end{equation}
define 
\begin{equation}
\begin{split}
    \Omega^+_{T_{\bfs_j}, m_{j+1}, s_{j+1}}:=& 
    \set{
    \bfeta=\mathcal{D}_{\bfm_j, \bfs_j}(\blambda)\cdot \pnorm{
    \bgamma'_{\bfm_j, \bfs_j}(\theta), \dots, \bgamma_{\bfm_j, \bfs_j}^{(n)}(\theta)
    }: \\
    & \bfeta\in \Omega^+_{T_{\bfs_j}, m_{j+1}}, \frac{|\lambda_{m_{j+1}+1}|}{s_{j+1}}+\dots+\frac{|\lambda_{m}|}{s_{j+1}^{m-m_{j+1}}}\simeq 1
    },
\end{split}
\end{equation}
with $\simeq$ replaced by $\lesim$ when $\ge $ in \eqref{221024e4_15} is an equality. We define $\phiplus_{T_{\bfs_j}, m_{j+1}, s_{j+1}}$ to be a frequency projection of $\phiplus_{T_{\bfs_j}}$ to the region $\Omega^+_{T_{\bfs_j}, m_{j+1}, s_{j+1}}$ so that 
\begin{equation}\label{221027e4_17}
    \phiplus_{T_{\bfs_j}}=\sum_{m\ge m_{j+1}>m_j}
    \sum_{s_{j+1}}
    \phiplus_{T_{\bfs_j}, m_{j+1}, s_{j+1}}.
\end{equation}
\begin{claim}\label{221026claim4_1}
Under the above notation, we have 
\begin{equation}
    \begin{split}
    &   \Norm{
    \sum_{\theta\in \Lambda_{\ell_j^{-1}\delta}}
    \sum_{T\in \W_{\theta}(\bfs_j; \Theta_j)}
    \phiplus_{T_{\bfs_j}, m_{j+1}, s_{j+1}}
    }_p
     \lessapprox
    \sum_{m_{j+1}}
    \sum_{s_{j+1}}
    \pnorm{
    \frac{1}{s_{j+1}}
    }^
    {1-
    \frac{
    \mathfrak{D}_
    {n_{j+1}-m_{j+1}}
    }
    {p}
    }\\
    & 
    \pnorm{
    \sum_{\ell(\Theta_{j+1})=s_{j+1}}
    \Norm{
    \sum_{\theta\in \Lambda_{\ell_j^{-1}\delta}\cap \Theta_{j+1}}
    \sum_{T\in \W_{\theta}(\bfs_j; \Theta_j)}
    \phiplus_{T_{\bfs_j}, m_{j+1}, s_{j+1}}
    }_p^p
    }^{\frac{1}{p}}
    \end{split}
\end{equation}
\end{claim}
\begin{proof}[Proof of Claim \ref{221026claim4_1}] 
The proof of this claim is essentially the same as that of Claim \ref{221027claim3_1}, and is therefore left out. 
\end{proof}
After this step of decoupling, we will apply a similar scaling to the one in \eqref{220901e3_13}. To explain this scaling, let us take $\Theta_{j+1}=[0, s_{j+1}]$; other intervals can be handled similarly. Let $\bxi$ denote a point in the Fourier support of 
\begin{equation}
    \sum_{\theta\in \Lambda_{\ell_j^{-1}\delta}\cap \Theta_{j+1}}
    \sum_{T\in \W_{\theta}(\bfs_j; \Theta_j)}
    \phiplus_{T_{\bfs_j}, m_{j+1}, s_{j+1}}.
\end{equation}
We apply the change of variables 
\begin{equation}\label{221025e4_20}
    \xi_1\mapsto \eta_1, \ \dots, \xi_{m_{j+1}}\mapsto \eta_{m_{j+1}}, \ \xi_{m_{j+1}+1}\mapsto s_{j+1} \eta_{m_{j+1}+1}, \dots, \xi_n\mapsto s_{j+1}^{n-m_{j+1}}\eta_n. 
\end{equation}
Let us make a remark here that in \eqref{221025e4_20} we indeed have multiple ways of doing changes of variables. For instance, another option would be
\begin{equation}
    \begin{split}
        & \xi_1\mapsto \eta_1, \ \dots, \xi_{m_{j+1}}\mapsto \eta_{m_{j+1}}, \ \xi_{m_{j+1}+1}\mapsto s_{j+1} \eta_{m_{j+1}+1}, \dots, \xi_{n_{j+1}}\mapsto s_{j+1}^{n_{j+1}-m_{j+1}}\eta_{n_{j+1}}, \\
        & \xi_{n_{j+1}+1}\mapsto \eta_{n_{j+1}+1}, \dots, \xi_{n}\mapsto \eta_n. 
    \end{split}
\end{equation}
In other words, we keep all the variables $\xi_{\iota}$ with $\iota> n_{j+1}$ unchanged. The idea is that, after arriving at the $(j+1)$-th step of the algorithm, all the frequency directions $\xi_{\iota}$ with $\iota> n_{j+1}$ will no longer play any role in any further decoupling inequalities, and instead we will simply use Fubini theorem along those directions. Here to make our notation consistent with previous ones, we still use the scaling in \eqref{221025e4_20}. \\

After the above change of variables, we define 
\begin{equation}
    \begin{split}
        & \bfs_{j+1}:=(s_1, \dots, s_j, s_{j+1}),\ \ \bfm_{j+1}=(m_1, \dots, m_j, m_{j+1}),\\
        & \thick_{j+1}:=
        \pnorm{
    \delta \prod_{j'=1}^{j+1} s_{j'}^{-(m+1-m_{j'})}, \dots, \delta \prod_{j'=1}^{j+1} s_{j'}^{-(n-m_{j'})}
    },
    \end{split}
\end{equation}
and also define
\begin{equation}
    \widehat{\phi}^{+}_{T_{\bfs_{j+1}}}(\bfeta):=\widehat{
    \phi
    }^+_{T_{\bfs_j}, m_{j+1}, s_{j+1}}(\eta_1, \dots, \eta_{m_{j+1}}, 
    s_{j+1} \eta_{m_{j+1}+1}, \dots, 
    s_{j+1}^{n-m_{j+1}}\eta_n
    ).
\end{equation}
Note from \eqref{221024e4_15} we see that the relation \eqref{221026e4_3aa} is upgraded to \begin{equation}
    \prod_{j'=1}^{j+1} s_{j'}^{n_{j+1}-m_{j'}}\le \delta;
\end{equation}
here we used the fact that $n_j=n_{j+1}$. Moreover, we combine \eqref{221024e4_4} and Claim \ref{221026claim4_1}, and see that \eqref{221024e4_4} holds with $j$ replaced by $j+1$, with $\ell_{j+1}:=\ell_j\cdot s_{j+1}$. For a frequency point $\bfeta$ in the support of $\widehat{\phi}^+_{T_{\bfs_{j+1}}}$, one can check directly that it is of the form \eqref{221026e4_6hh}--\eqref{221026e4_10hh}, with every $j$ there replaced by $j+1$.  In the end, we let $\Omega^+_{T_{\bfs_{j+1}}}$ denote the support of $\widehat{\phi}^+_{T_{\bfs_{j+1}}}$, and finish the analysis for the former case in \eqref{221024e4_11}.\\

We consider the latter case in \eqref{221024e4_11}. The analysis for this case is similar to the former case, and we only explain the differences. Recall from \eqref{220902e3_24zzz} that 
\begin{equation}
    |\lambda_{m_j+1}|+\dots+|\lambda_m|\lesim 1. 
\end{equation}
For integers $m_{j+1}$ satisfying $m\ge m_{j+1}\ge m_j$, define
\begin{equation}
    \begin{split}
        \Omega^+_{T_{\bfs_j}, m_{j+1}}:=
        & \Big\{\bfeta=\mathcal{D}_{\bfm_j, \bfs_j}(\blambda)\cdot \pnorm{
    \bgamma'_{\bfm_j, \bfs_j}(\theta), \dots, \bgamma_{\bfm_j, \bfs_j}^{(n)}(\theta)
    }: \\
    & T\in \W_{\theta}(\bfs_j; \Theta_j), \bfeta\in \Omega^+_{T_{\bfs_j}}, |\lambda_{m_{j+1}}|\simeq 1, |\lambda_{\iota}|\ll 1, \forall \iota>m_{j+1}
    \Big\}
    \end{split}
\end{equation}
Note that here $m_{j+1}$ could be equal to $m_j$, which is different the previous case.  Next, for positive dyadic numbers $s_{j+1}$ with 
\begin{equation}\label{221024e4_15dd}
    1\gg s_{j+1}\ge \pnorm{
    \delta
    \prod_{j'=1}^j s_{j'}^{-(n_j-1-m_{j'})}
    }^{\frac{1}{n_{j+1}-m_{j+1}}},
\end{equation}
define 
\begin{equation}
\begin{split}
    \Omega^+_{T_{\bfs_j}, m_{j+1}, s_{j+1}}:=& 
    \set{
    \bfeta=\mathcal{D}_{\bfm_j, \bfs_j}(\blambda)\cdot \pnorm{
    \bgamma'_{\bfm_j, \bfs_j}(\theta), \dots, \bgamma_{\bfm_j, \bfs_j}^{(n)}(\theta)
    }: \\
    & \bfeta\in \Omega^+_{T_{\bfs_j}, m_{j+1}}, \frac{|\lambda_{m_{j+1}+1}|}{s_{j+1}}+\dots+\frac{|\lambda_{m}|}{s_{j+1}^{m-m_{j+1}}}\simeq 1
    },
\end{split}
\end{equation}
with $\simeq$ replaced by $\lesim$ when $\ge $ in \eqref{221024e4_15dd} is an equality. If we compare \eqref{221024e4_15dd} with \eqref{221024e4_15}, then we see  that $n_j$ in \eqref{221024e4_15} is replaced by $n_j-1=n_{j+1}$. The principle behind it is that we always pick the right most component in the vector $\thick_j$ that is strictly smaller than one. The rest of the argument remains the same.

\section{Output of the algorithm}

The above algorithm outputs $J$-tuples of integers $\bfm_J=(m_1, \dots, m_J)$ with 
\begin{equation}
    m_1\le \dots\le m_J\le m,
\end{equation}
$J$-tuples of dyadic numbers $\bfs_J=(s_1, \dots, s_J)$, dimensions parameters $(n_1, \dots, n_J)$ satisfying 
\begin{equation}
    n_1\ge \dots\ge n_J>m,
\end{equation}
for each $J=1, 2, \dots, \mathfrak{J}$, with $|\mathfrak{J}|\lesim_n 1$. The bound for $\mathfrak{J}$ follows from the fact that after each step of running our algorithm, either the dimension parameter $n_j$ reduces by one, or the parameter $m_{j}$ increases by one.  Moreover, 
\begin{equation}\label{221027e5_3}
\begin{split}
        \Norm{
    \sum_T \phi_T^+
    }_p \lessapprox & 
    \sum_{J=1}^{\mathfrak{J}}
    \sum_{\bfm_J} \sum_{\bfs_J} 
    \pnorm{
    \prod_{j=1}^J 
    \jac(s_{j})
    }
    \pnorm{
    \prod_{j=1}^J 
    \pnorm{
    \frac{1}{s_{j}}
    }^{1-\frac{\mathfrak{D}_
    {n_{j}-m_{j}}}{p}}
    }\\
    & 
    \pnorm{
    \sum_{\ell(\Theta_J)=\ell_J}
    \Norm{
    \sum_{\theta\in \Lambda_{\ell_J^{-1}\delta}}
    \sum_{T\in \W_{\theta}(\bfs_J; \Theta_J)}
    \phiplus_{T_{\bfs_J}}
    }_p^p
    }^{1/p}.
\end{split}
\end{equation}
That the algorithm terminates after producing the parameters $\bfm_J, \bfs_J$, etc. means that $n_J=m+1$, and 
\begin{equation}\label{221027e5_4zz}
    \prod_{j=1}^J s_{j}^{n_J-m_{j}}=\delta. 
\end{equation}
For the term $\phiplus_{T_{\bfs_J}}$, we revert all the changes of variables as in \eqref{221025e4_20} and \eqref{220901e3_13}, and write the resulting function as $\phiplus_T* \widecheck{\psi_{\tau, \bfs_J}}$, where $\tau=\tau(T)$ is the dual of $T$, and we use $\psi_{\tau, \bfs_J}$ to collect all the frequency cut-offs in \eqref{221027e4_17}. In other words, each $\psi_{\tau, \bfs_J}$ is supported on a dyadic rectangular box that is a subset of $\tau$. The estimate \eqref{221027e5_3} becomes 
\begin{equation}\label{221027e5_5}
\begin{split}
        \Norm{
    \sum_T \phi_T^+
    }_p \lessapprox & 
    \sum_{J=1}^{\mathfrak{J}}
    \sum_{\bfm_J} \sum_{\bfs_J} 
    \pnorm{
    \prod_{j=1}^J 
    \pnorm{
    \frac{1}{s_{j}}
    }^{1-\frac{\mathfrak{D}_
    {n_{j}-m_{j}}}{p}}
    }\\
    & 
    \pnorm{
    \sum_{\ell(\Theta_J)=\ell_J}
    \Norm{
    \sum_{\theta\in \Lambda_{\delta}\cap \Theta_J}
    \sum_{T\in \W_{\theta}}
    \phiplus_{T}* \widecheck{\psi_{\tau, \bfs_J}}
    }_p^p
    }^{1/p}.
\end{split}
\end{equation}
Let us write down again the family of relations in \eqref{221026e4_3aa}
\begin{equation}
    \prod_{j'=1}^j s_{j'}^{-(n_j-m_{j'})}\le \delta^{-1},
\end{equation}
which holds for every $J$-tuple $\bfm_J, \bfs_J$ and every $j\le J$. Moreover, if at a given $j$, it holds that $n_{j+1}=n_j-1$, then 
\begin{equation}\label{221027e5_7ppp}
    \prod_{j'=1}^j s_{j'}^{n_j-m_{j'}}=\delta,
\end{equation}
that is, we have equality in the previous relation.

We continue to bound \eqref{221027e5_5}. At this point, we do not have any non-trivial decoupling to use, and we simply use the triangle  inequality. Note that \eqref{221027e5_4zz} implies that $\ell_J\ge \delta$. By the triangle inequality, 
\begin{equation}\label{221027e5_5zzz}
\begin{split}
        \Norm{
    \sum_T \phi_T^+
    }_p \lessapprox & 
    \sum_{J=1}^{\mathfrak{J}}
    \sum_{\bfm_J} \sum_{\bfs_J} 
    \pnorm{
    \prod_{j=1}^J 
    \pnorm{
    \frac{1}{s_{j}}
    }^{1-\frac{\mathfrak{D}_
    {n_{j}-m_{j}}}{p}}
    }
    \pnorm{\ell_J \delta^{-1}}^{1-\frac{1}{p}}
    \pnorm{
    \sum_{\theta\in \Lambda_{\delta}}
    \Norm{
    \sum_{T\in \W_{\theta}}
    \phiplus_{T}* \widecheck{\psi_{\tau, \bfs_J}}
    }_p^p
    }^{1/p}.
\end{split}
\end{equation}
To bound the last $L^p$, note that $\widecheck{\psi_{\tau, \bfs_J}}$ is independent of $T\in \W_{\theta}$ for each fixed $\theta$. Therefore, by Young's inequality, 
\begin{equation}
    \begin{split}
        \Norm{
    \sum_T \phi_T^+
    }_p \lessapprox & 
    \sum_{J=1}^{\mathfrak{J}}
    \sum_{\bfm_J} \sum_{\bfs_J} 
    \pnorm{
    \prod_{j=1}^J 
    \pnorm{
    \frac{1}{s_{j}}
    }^{1-\frac{\mathfrak{D}_
    {n_{j}-m_{j}}}{p}}
    }
    \pnorm{\ell_J \delta^{-1}}^{1-\frac{1}{p}}
    \pnorm{
    \sum_{\theta\in \Lambda_{\delta}}
    \Norm{
    \sum_{T\in \W_{\theta}}
    \phiplus_{T}
    }_p^p
    }^{1/p}.
\end{split}
\end{equation}
Recall that we need to prove \eqref{221027e2_16ppp}. It therefore remains to prove that \begin{equation}
    \Big(\prod_{j=1}^J \pnorm{\frac{1}{s_j}}^{p-\mathfrak{D}_{n_j-m_j}}\Big) 
    \pnorm{
    \frac{\ell_J}{\delta}
    }^{p-1}
    \delta^{m-n} \lesim \delta^{-p+1},
\end{equation}
which is equivalent to 
\begin{equation}
    \prod_{j=1}^J s_j^{\mathfrak{D}_{n_j-m_j}-1} \lesim \delta^{n-m}. 
\end{equation}
By the definition of $\mathfrak{D}_{\iota}$, it suffices to prove 
\begin{equation}
    \prod_{j=1}^J s_j^{1+\dots+(n_j-m_j)}\lesim \delta^{n-m}. 
\end{equation}
However, this follows immediately from multiplying all the $(n-m)$ identities in \eqref{221027e5_7ppp}.

\normalem

\end{document}